\documentclass[a4paper, 12pt,oneside,reqno]{amsart}
\usepackage[a4paper]{geometry}
\usepackage{amssymb,amsmath,amsthm}
\usepackage[T2A]{fontenc}
\usepackage[breaklinks=true]{hyperref}

\usepackage{xcolor}

\usepackage{bbm}
\usepackage{graphicx}

\def\Var{\mathrm{Var}}
\def\dim{\mathrm{dim}}

\let\<\langle
\let\>\rangle

\newtheorem{lemma}{Lemma}
\newtheorem{proposition}{Proposition}
\newtheorem{theorem}{Theorem}

\newtheorem{remark}{Remark}

\title{Malcev classification for the variety of left-symmetric algebras}

\author{A. Ryskeldin}

\address{SDU University, Kaskelen, Kazakhstan}

\email{ryskeldin@gmail.com}

\author{B. Sartayev$^{*}$}

\address{Narxoz University, Almaty, Kazakhstan and SDU University, Kaskelen, Kazakhstan}

\email{baurjai@gmail.com}



\thanks{${}^{*}$Corresponding author: Bauyrzhan Sartayev   (baurjai@gmail.com)}

\begin{document}

\maketitle

\begin{abstract}
In this paper, we study three classes of subvarieties inside the variety of left-symmetric algebras. We show that these subvarieties are naturally related to some well-known varieties, such as alternative, assosymmetric and Zinbiel algebras. For certain subvarieties of the varieties of alternative and assosymmetric algebras, we explicitly construct bases of the corresponding free algebras.
We then define the commutator and anti-commutator operations on these algebras and derive a number of identities satisfied by these operations in all degrees up to $4$.
\end{abstract}

\section{Introduction}

The class of left-symmetric algebras (also called pre-Lie algebras) arose in differential geometry and deformation theory \cite{Gerst,Kosz,Vin}. Recall that a left-symmetric algebra is a vector space \(A\) with a bilinear product \((a,b) \mapsto ab\) such that its associator
\[
(a,b,c) = (ab)c - a(bc)
\]
is symmetric in the first two arguments, that is,
\[
(a,b,c) = (b,a,c)
\]
for all \(a,b,c \in A\).

In this paper, we study three types of left-symmetric algebras based on Malcev's classification. As shown in \cite{Mal'cev}, the space of invariant relations of degree $3$ decomposes as a sum
\[
\mathfrak{L}=\mathfrak{L}_1+\mathfrak{L}_2,
\]
where $\mathfrak{L}_1:=\mathfrak{A}_1+\mathfrak{B}_1+\mathfrak{C}_1+\mathfrak{D}_1$, and the subspaces are defined by the following polynomial relations:
\begin{gather}
\mathfrak{A}_1:\;(ab)c+(ba)c+(ac)b+(ca)b+(bc)a+(cb)a=0,
\label{as1} \\
\mathfrak{B}_1:\;(ab)c-(ba)c-(ac)b+(ca)b+(bc)a-(cb)a=0,
\label{as2}
\end{gather}
\begin{gather}\label{as3}
\mathfrak{C}_1:\;
\begin{cases}
(ab)c+(ba)c-(bc)a-(cb)a=0, \\
(ac)b+(ca)b-(cb)a-(bc)a=0,
\end{cases}
\end{gather}
\begin{gather}\label{as4}
\mathfrak{D}_1:\;
\begin{cases}
(ab)c+(ac)b-(cb)a-(ca)b=0, \\
(ab)c+(ac)b-(bc)a-(ba)c=0.
\end{cases}
\end{gather}
Analogously, set $\mathfrak{L}_2:=\mathfrak{A}_2+\mathfrak{B}_2+\mathfrak{C}_2+\mathfrak{D}_2$, defined by the same relations written with right-normed bracketing. We denote by $\mathcal{L}\mathcal{S}_{\mathfrak{Z}_i}$ the variety of left-symmetric algebras with identity $\mathfrak{Z}_i$, where $\mathfrak{Z}=\mathfrak{A},\mathfrak{B},\mathfrak{C},\mathfrak{D}$ and $i=1,2$. The Malcev classification was studied for the varieties of associative and alternative algebras, see \cite{KazMat,MalcAs1,MalcAs2}. In \cite{MalcAs1} it was shown that, in the associative case, the operads obtained from the Malcev classification via the Koszul dual computation are exactly the alternative, assosymmetric and $(-1,1)$ operads. For further results on the importance of these operads, see \cite{assos1,Dzhuma,As1,As2,binaryperm,-1-1}.

We restrict ourselves to the varieties
\[
\mathcal{L}\mathcal{S}_{\mathcal{A}_1},\mathcal{L}\mathcal{S}_{\mathcal{B}_1} \;\textrm{and}\; \mathcal{L}\mathcal{S}_{\mathcal{A}_2}.
\]
There is no need to consider $\mathcal{L}\mathcal{S}_{\mathcal{B}_2}$ separately, since it coincides with $\mathcal{L}\mathcal{S}_{\mathcal{B}_1}$. Moreover, we shall see that the varieties $\mathcal{L}\mathcal{S}_{\mathcal{A}_1}$, $\mathcal{L}\mathcal{S}_{\mathcal{B}_1}$ and $\mathcal{L}\mathcal{S}_{\mathcal{A}_2}$
are pairwise distinct.

For each variety $\mathcal{L}\mathcal{S}_{\mathfrak{Z}_i}$ under consideration, we determine its Koszul dual operad in the sense of \cite{GK94}. We then describe all polynomial identities of degrees up to~$4$ satisfied by the commutator and anti-commutator operations. Moreover, we construct an explicit basis for each of the obtained Koszul dual operads. Since the Koszul dual of the left-symmetric operad is a perm operad (associative operad with a left-commutative identity), for $\mathcal{L}\mathcal{S}_{\mathfrak{Z}_i}$ and $\mathcal{L}\mathcal{S}_{\mathfrak{Z}_i}^{\;(!)}$, we describe the following diagram:
\begin{gather}\label{diagram}
\begin{picture}(400,90)
  \put(120,75){$\mathcal{L}\mathcal{S}$}
  \put(255,75){$\mathcal{\mathcal{P}}erm$}  
  \put(125,57){$\cup$}
  \put(256,57){$\cap$}
  \put(120,40){$\mathcal{L}\mathcal{S}_{\mathfrak{Z}_i}$}
  \put(250,40){$\mathcal{L}\mathcal{S}_{\mathfrak{Z}_i}^{\;(!)}$}
  \put(256,20){$\cap$}
  \put(256,5){$\mathcal{A}$}
  \put(175,43){\vector(1,0){55}}
  \put(210,43){\vector(-1,0){55}}
  \put(175,79){\vector(1,0){55}}
  \put(210,79){\vector(-1,0){55}}
\end{picture}
\end{gather}
Artin's theorem states that every alternative algebra generated by two elements is associative, i.e., the variety of binary associative algebras coincides with the variety of alternative algebras \cite{Artin}. Also, it was shown in \cite{binaryperm} that the variety of binary perm algebras (which we denote by $\mathcal{P}erm_2$) coincides with the variety of alternative algebras satisfying an additional identity of degree $3$, which is
\[
a(bc)+a(cb)=b(ac)+c(ab).
\]
In the cases $\mathfrak{Z}_i = \mathfrak{A}_1$ and $\mathfrak{Z}_i = \mathfrak{A}_2$, it turns out that the operad $\mathcal{L}\mathcal{S}_{\mathfrak{Z}_i}^{\;(!)}$ is an alternative satisfying an additional identity of degree $3$.
Moreover, in the case $\mathfrak{Z}_i = \mathfrak{A}_1$, the additional identity is again left-commutative. From the general observations, we obtain
\[
\mathcal{P}erm\subset \mathcal{L}\mathcal{S}_{\mathfrak{A}_1}^{\;(!)}\subset\mathcal{P}erm_2.
\]
Although this is not immediate, direct computations show that
\[
\mathcal{P}erm\subset \mathcal{L}\mathcal{S}_{\mathfrak{A}_2}^{\;(!)}\subset\mathcal{P}erm_2.
\]
Also, we shall see that
\[
\mathcal{P}erm\subset \mathcal{L}\mathcal{S}_{\mathfrak{B}_1}^{\;(!)}\not\subset\mathcal{P}erm_2.
\]
To compute the Koszul dual operad of $\mathcal{L}\mathcal{S}_{\mathfrak{Z}_i}$, we need the multilinear basis of degree $3$ of the algebra $\mathcal{L}\mathcal{S}\<a,b,c\>$. We use the following rewriting rules:
\[
a(bc)=(ab)c-(ba)c+b(ac),
\]
\[
a(cb)=(ac)b-(ca)b+c(ab)
\]
and
\[
b(ca)=(bc)a-(cb)a+c(ba).
\]
Let us compute the polarization of $\mathcal{L}\mathcal{S}$ following \cite{MarklRemm}; this construction will be used throughout the paper.

\begin{proposition}\label{polarization}
The polarization of $\mathcal{L}\mathcal{S}\<X\>$ is
\begin{gather}\label{pol1}
[[a,b],c]+[[b,c],a]+[[c,a],b]=0    
\end{gather}
and
\begin{gather}\label{pol2}
\{\{a,b\},c\}=-\{[a,b],c\}-2\{[a,c],b\}+[\{a,b\},c]-[[a,c],b]+\{a,\{b,c\}\}-\{a,[b,c]\}+[a,\{b,c\}].
\end{gather}
\end{proposition}
\begin{proof}
Firstly, we rewrite each identity of $\mathcal{LS}$ using commutators and anti-commutators as follows:
\[
ab=1/2([a,b]+\{a,b\}).
\]
We define an order on monomials with operations $[\cdot,\cdot]$ and $\{\cdot,\cdot\}$ by
\begin{itemize}
    \item $[[\cdot,\cdot],\cdot]>[\{\cdot,\cdot\},\cdot]>\{[\cdot,\cdot],\cdot\}>\{\{\cdot,\cdot\},\cdot\}$;
    \item the remaining monomials are ordered by lexicographical order;
\end{itemize}
Let $\mathcal{R}$ be a set of relations in $\mathcal{LS}$ in terms of $[\cdot,\cdot]$ and $\{\cdot,\cdot\}$, i.e., we have
\begin{multline*}
(ab)c-a(bc)+b(ac)-(ba)c=1/4([[a,b],c]+[\{a,b\},c]+\{[a,b],c\}+\{\{a,b\},c\}\\
+[[b,c],a]-\{[b,c],a\}+[\{b,c\},a]-\{\{b,c\},a\}\\
+[[a,b],c]-[\{a,b\},c]+\{[a,b],c\}-\{\{a,b\},c\}\\
-[[a,c],b]+\{[a,c],b\}-[\{a,c\},b]+\{\{a,c\},b\}),
\end{multline*}
\begin{multline*}
(ac)b-a(cb)-(ca)b+c(ab)=1/4([[a,c],b]+[\{a,c\},b]+\{[a,c],b\}+\{\{a,c\},b\}\\
-[[b,c],a]+\{[b,c],a\}+[\{b,c\},a]-\{\{b,c\},a\}\\
+[[a,c],b]-[\{a,c\},b]+\{[a,c],b\}-\{\{a,c\},b\}\\
-[[a,b],c]+\{[a,b],c\}-[\{a,b\},c]+\{\{a,b\},c\})),\\  
\end{multline*}
and
\begin{multline*}
(bc)a-b(ca)-(cb)a+c(ba)=1/4([[b,c],a]+[\{b,c\},a]+\{[b,c],a\}+\{\{b,c\},a\}\\
-[[a,c],b]+\{[a,c],b\}+[\{a,c\},b]-\{\{a,c\},b\}\\
+[[b,c],a]-[\{b,c\},a]+\{[b,c],a\}-\{\{b,c\},a\}\\
+[[a,b],c]-\{[a,b],c\}-[\{a,b\},c]+\{\{a,b\},c\}).\\
\end{multline*}
For the set $\mathcal{R}$, we construct a matrix $[\mathcal{R}]$ representing the identities as row vectors relative to the given order.
\[ \left(
\begin{array}{cccccccccccc}
2/4 & -1/4 &  1/4 &  0 & -1/4 &  1/4 &  2/4 &  1/4 & -1/4 &  0 &  1/4 & -1/4\\
-1/4 &  2/4 & -1/4 & -1/4 &  0 &  1/4 &  1/4 &  2/4 &  1/4 &  1/4 &  0 & -1/4\\
 1/4 & -1/4 &  2/4 & -1/4 &  1/4 &  0 & -1/4 &  1/4 &  2/4 &  1/4 & -1/4 &  0
\end{array} \right)
\]
By applying Gröbner basis theory, the matrix $[\mathcal{R}]$ can be brought to an echelon form, and we obtain the result.
\end{proof}

Since there is a one-to-one correspondence between a variety of algebras
defined by identities of degrees two and three $\Var$ and the associated quadratic operad,
we will use the same terminology for both objects throughout the paper. We consider all algebras over a field $K$ of characteristic $0$.

\section{$\mathfrak{A}_1$-left-symmetric algebras}

In this section, we consider algebras from the variety $\mathcal{L}\mathcal{S}_{\mathfrak{A}_1}$.  Let us compute the dual operad for $\mathcal{L}\mathcal{S}_{\mathfrak{A}_1}$. The Lie-admissibility condition for $S\otimes U$ gives the defining identities of the operad $\mathcal{L}\mathcal{S}_{\mathfrak{A}_1}^{\;\;(!)}$, where $S$ is an algebra from $\mathcal{L}\mathcal{S}_{\mathfrak{A}_1}$. So, we have
\begin{multline*}
[[a\otimes u,b\otimes v],c\otimes w]=(ab)c\otimes (uv)w-(ba)c\otimes (vu)w-
c(ab)\otimes w(uv)+c(ba)\otimes w(vu)=\\
(-(ba)c-(ac)b-(ca)b-(bc)a-(cb)a)\otimes (uv)w-(ba)c\otimes (vu)w-
c(ab)\otimes w(uv)+c(ba)\otimes w(vu),
\end{multline*}
\begin{multline*}
[[b\otimes v,c\otimes w],a\otimes u]=(bc)a\otimes (vw)u-(cb)a\otimes (wv)u-
a(bc)\otimes u(vw)+a(cb)\otimes u(wv)=\\
(bc)a\otimes (vw)u-(cb)a\otimes (wv)u-
(-2(ba)c-(ac)b-(ca)b-(bc)a-(cb)a+b(ac))\otimes u(vw)\\
+((ac)b-(ca)b+c(ab))\otimes u(wv)
\end{multline*}
and
\begin{multline*}
[[c\otimes w,a\otimes u],b\otimes v]=(ca)b\otimes (wu)v-(ac)b\otimes (uw)v-
b(ca)\otimes v(wu)+b(ac)\otimes v(uw)=\\
 (ca)b\otimes (wu)v-(ac)b\otimes (uw)v-
((bc)a-(cb)a+c(ba))\otimes v(wu)+b(ac)\otimes v(uw).
\end{multline*}
The sum of the same elements on the left side of the tensors yields the following result: 
\begin{proposition}
The operad $\mathcal{L}\mathcal{S}_{\mathfrak{A}_1}^{\;\;(!)}$ is the alternative operad with a left-commutative identity.
\end{proposition}

\begin{proposition}
The polarization of $\mathcal{L}\mathcal{S}_{\mathfrak{A}_1}\<X\>$ is defined by (\ref{pol1}), (\ref{pol2}) and
\begin{multline*}
\{a,\{b,c\}\}=\{[a,b],c\}+\{[a,c],b\}-3/2[\{a,b\},c]-3/2[\{a,c\},b]\\
+3/2[[a,c],b]-1/3[a,\{b,c\}]+1/3[a,[b,c]].
\end{multline*}
\end{proposition}
\begin{proof}
The proof is analogous to Proposition \ref{polarization}.
\end{proof}

Also, from direct computations, we obtain the following results:
\begin{theorem}
All identities of the algebra $\mathcal{L}\mathcal{S}_{\mathfrak{A}_1}^{\;(-)}\<X\>$ up to degree $4$ follow from the anti-commu\-tativity and Jacobi identities.
\end{theorem}
\begin{theorem}
All identities of the algebra $\mathcal{L}\mathcal{S}_{\mathfrak{A}_1}^{\;(+)}\<X\>$ up to degree $4$ follow from the commutati\-vity identity.
\end{theorem}

Let $\mathcal{L}\mathcal{S}_{\mathfrak{A}_1}^{\;\;(!)}\<X\>$ be a free algebra of the variety $\mathcal{L}\mathcal{S}_{\mathfrak{A}_1}^{\;\;(!)}$.
\begin{lemma}\label{lemmaA1}
In algebra $\mathcal{L}\mathcal{S}_{\mathfrak{A}_1}^{\;\;(!)}\<X\>$ the following identities hold:
\[
((ba)c)d=((ac)b)d,
\]
\[
(((ab)c)d)e=(((ac)b)d)e
\]
and
\[
(((ab)c)d)e=(((ab)d)c)e.
\]
\end{lemma}
\begin{proof}
It can be proved using computer algebra as software programs  Albert \cite{Albert}. 
\end{proof}

Let $\mathfrak{A}_1^{(!)}(i)$ be a set which is defined for each $i$ as follows:
\[
\mathfrak{A}_1^{(!)}(1)=\{x_{k_1}\},\;\mathfrak{A}_1^{(!)}(2)=\{x_{k_1}x_{k_2}\},
\]
\[
\mathfrak{A}_1^{(!)}(3)=\{
(x_{l_1}x_{l_2})x_{l_3},(x_{l_1}x_{l_3})x_{l_2},
(x_{l_2}x_{l_1})x_{l_3},
x_{l_3}(x_{l_2}x_{l_1}),
\},
\]
\[
\mathfrak{A}_1^{(!)}(4)=\{((x_{s_1}x_{s_2})x_{s_3})x_{s_4},((x_{s_1}x_{s_2})x_{s_4})x_{s_3},((x_{s_1}x_{s_3})x_{s_4})x_{s_2},((x_{s_2}x_{s_1})x_{s_3})x_{s_4},((x_{s_2}x_{s_3})x_{s_4})x_{s_1}
\}
\]
and
\[
\mathfrak{A}_1^{(!)}(n)=\{((\cdots((x_{i_1}x_{i_2})x_{i_3})\cdots )x_{i_{n-1}})x_{i_n}\},
\]
where $l_1\leq l_2\leq l_3$, $s_1\leq s_2\leq s_3\leq s_4$, $i_1\leq i_2\leq\cdots \leq i_{n-1}$ and $n\geq 5$. Also, we define 
\[
\overline{\mathfrak{A}^{(!)}}_1=\bigcup_{i\geq 1}\mathfrak{A}_1^{(!)}(i)
\]
\begin{theorem}
The set $\overline{\mathfrak{A}^{(!)}}_1$ is a basis of algebra $\mathcal{L}\mathcal{S}_{\mathfrak{A}_1}^{\;\;(!)}\<X\>$.
\end{theorem}
\begin{proof}
Firstly, we show that the set $\overline{\mathfrak{A}^{((!))}}_1$ spans $\mathcal{L}\mathcal{S}_{\mathfrak{A}_1}^{\;\;(!)}\<X\>$. For degrees up to $4$, the result can be verified by software programs, such as Albert \cite{Albert}. The left-alternative and left-commutative identities give that
\[
a(bc)=1/2((ba)c+(ab)c),
\]
i.e., every monomial in $\mathcal{L}\mathcal{S}_{\mathfrak{A}_1}^{\;\;(!)}\<X\>$ can be written as a linear combination of left-normed monomials. The identities from Lemma \ref{lemmaA1} provide an ordering of all generators except the external one in a left-normed monomial, i.e.,
\[
((\cdots((x_{i_1}x_{i_2})x_{i_3})\cdots )x_{i_{n-1}})x_{i_n}=((\cdots((x_{\sigma (i_1)}x_{\sigma (i_2)})x_{\sigma (i_3)})\cdots )x_{\sigma (i_{n-1})})x_{i_n},
\]
where $\sigma\in S_{n-1}$.

To prove the linear independence of the set $\overline{\mathfrak{A}^{((!))}}_1$, we construct the multiplication table for the algebra $A_1\<X\>$ with the basis $\overline{\mathfrak{A}^{((!))}}_1$ and show that such an algebra $A_1\<X\>$ belongs to the variety $\mathcal{L}\mathcal{S}_{\mathfrak{A}_1}^{\;\;(!)}\<X\>$. Up to degree 4, we define multiplication in $A_1\<X\>$ that is consistent with the defining identities of $\mathcal{L}\mathcal{S}_{\mathfrak{A}_1}^{\;\;(!)}\<X\>$. Starting from degree $5$, the multiplication is defined as follows:
\begin{multline}\label{mult1}
(((\cdots((x_{i_1}x_{i_2})x_{i_3})\cdots )x_{i_{n-1}})x_{i_n})(((\cdots((x_{j_1}x_{j_2})x_{j_3})\cdots )x_{j_{m-1}})x_{j_m})=\\
((\cdots((x_{k_1}x_{k_2})x_{k_3})\cdots )x_{k_{n+m-1}})x_{j_{m}},
\end{multline}
where $\{k_1,k_2,\ldots,k_{n+m-1}\}=\{i_1,i_2,\ldots,i_{n},j_1,j_2,\ldots,j_{m-1}\}$ and $k_1\leq k_2\leq\ldots\leq k_{n+m-1}$. It remains to check that the algebra $A_1\<X\>$ satisfies the left-alternative, right-alternative, and left-commutative identities.

Let
\[
a = ((\cdots((x_{i_1}x_{i_2})x_{i_3})\cdots)x_{i_{n-1}})x_{i_n},\quad
b = ((\cdots((y_{j_1}y_{j_2})y_{j_3})\cdots)y_{j_{m-1}})y_{j_m},
\]
\[
c = ((\cdots((z_{k_1}z_{k_2})z_{k_3})\cdots)z_{k_{t-1}})z_{k_t}
\]
be arbitrary left-normed monomials from the basis $\overline{\mathfrak{A}^{(!)}}_1$ of $A_1\<X\>$.  
We verify that the following identities hold in $A_1\<X\>$:
\[
(aa)b = a(ab),\qquad
(ab)b = a(bb),\qquad
a(bc) = a(bc).
\]
By (\ref{mult1}), both sides of the identity $(aa)b = a(ab)$ can be written as a left-normed monomial whose rightmost generator is $y_{j_m}$, while all other generators are ordered. Therefore, the two expressions coincide. For the same reason, we also obtain the equalities $(ab)b = a(bb)$ and $a(bc) = a(bc)$.

\end{proof}

For the operad $\mathcal{L}\mathcal{S}_{\mathfrak{A}_1}^{\;\;(!)}$, we have
\[
\dim(\mathcal{L}\mathcal{S}_{\mathfrak{A}_1}^{\;\;(!)}(1))=1,\;\;\dim(\mathcal{L}\mathcal{S}_{\mathfrak{A}_1}^{\;\;(!)}(2))=2,\;\;\dim(\mathcal{L}\mathcal{S}_{\mathfrak{A}_1}^{\;\;(!)}(3))=4,\;\;
\dim(\mathcal{L}\mathcal{S}_{\mathfrak{A}_1}^{\;\;(!)}(4))=5,
\]
and
\[
\dim(\mathcal{L}\mathcal{S}_{\mathfrak{A}_1}^{\;\;(!)}(n))=n,
\]
where $n\geq 5$.
In addition, using the computer algebra system Albert~\cite{Albert}, we obtain
\begin{center}
\begin{tabular}{c|cccccc}
 $n$ & 1 & 2 & 3 & 4 & 5 & 6 \\
 \hline $\dim(\mathcal{LS}_{\mathfrak{A}_1})(n)$ & 1 & 2 & 8 & 45 & 314 & 2499
\end{tabular}
\end{center}

As observed above, an algebra $\mathcal{LS}_{\mathfrak{A}_1}^{\;(!)}\langle X\rangle$ satisfies the identity
\[
a(bc) = \frac{1}{2}\bigl((ba)c + (ab)c\bigr).
\]
Hence, the variety $\mathcal{LS}_{\mathfrak{A}_1}^{\;(!)}$ can be represented as
\[
\mathcal{LS}_{\mathfrak{A}_1}^{\;(!)}
  = \mathcal{A}lt \;\cap\; \delta\textrm{-}\mathcal{Z}inb,
\]
that is, $\mathcal{LS}_{\mathfrak{A}_1}^{\;(!)}$ is the intersection of the varieties of alternative algebras and $\delta$-Zinbiel algebras, where $\delta = \frac{1}{2}$. For recent results on Zinbiel algebras, see \cite{Mash1,Mash2,KolMashSar}.

In diagram~(\ref{diagram}), in the case $\mathfrak{Z}_i = \mathfrak{A}_1$, the variety
$\mathcal{A}$ may be taken to be either the variety of alternative algebras or the variety
of $\tfrac{1}{2}$-Zinbiel algebras. Our computations show that, up to degree~$4$, the
classes of algebras $\mathcal{L}\mathcal{S}^{(-)}_{\mathfrak{Z}_i}$ and
$\mathcal{L}\mathcal{S}^{(+)}_{\mathfrak{Z}_i}$ coincide with the class of Lie and
commutative algebras, respectively.

\section{$\mathfrak{B}_1$-left-symmetric algebras}
In this section, we consider algebras from the variety $\mathcal{L}\mathcal{S}_{\mathfrak{B}_1}$.  Let us compute the dual operad for $\mathcal{L}\mathcal{S}_{\mathfrak{B}_1}$. The Lie-admissibility condition for $S\otimes U$ gives the defining identities of the operad $\mathcal{L}\mathcal{S}_{\mathfrak{B}_1}^{\;\;(!)}$, where $S$ is an algebra from $\mathcal{L}\mathcal{S}_{\mathfrak{B}_1}$. So, we have
\begin{multline*}
[[a\otimes u,b\otimes v],c\otimes w]=(ab)c\otimes (uv)w-(ba)c\otimes (vu)w-
c(ab)\otimes w(uv)+c(ba)\otimes w(vu)=\\
((ba)c+(ac)b-(ca)b-(bc)a+(cb)a)\otimes (uv)w-(ba)c\otimes (vu)w-
c(ab)\otimes w(uv)+c(ba)\otimes w(vu),
\end{multline*}
\begin{multline*}
[[b\otimes v,c\otimes w],a\otimes u]=(bc)a\otimes (vw)u-(cb)a\otimes (wv)u-
a(bc)\otimes u(vw)+a(cb)\otimes u(wv)=\\
(bc)a\otimes (vw)u-(cb)a\otimes (wv)u-
((ac)b-(ca)b-(bc)a+(cb)a+b(ac))\otimes u(vw)\\
+((ac)b-(ca)b+c(ab))\otimes u(wv)
\end{multline*}
and
\begin{multline*}
[[c\otimes w,a\otimes u],b\otimes v]=(ca)b\otimes (wu)v-(ac)b\otimes (uw)v-
b(ca)\otimes v(wu)+b(ac)\otimes v(uw)=\\
 (ca)b\otimes (wu)v-(ac)b\otimes (uw)v-
((bc)a-(cb)a+c(ba))\otimes v(wu)+b(ac)\otimes v(uw).
\end{multline*}
The sum of the same elements on the left side of the tensors, we obtain the following result: 
\begin{proposition}
The operad $\mathcal{L}\mathcal{S}_{\mathfrak{B}_1}^{\;\;(!)}$ is the assosymmetric operad with a left-commutative identity.
\end{proposition}

Also, from direct computations, we obtain the following results:
\begin{theorem}
All identities of the algebra $\mathcal{L}\mathcal{S}_{\mathfrak{B}_1}^{\;(-)}\<X\>$ up to degree $4$ follow from the anti-commuta\-tivity and Jacobi identities.
\end{theorem}

\begin{theorem}
All identities of the algebra $\mathcal{L}\mathcal{S}_{\mathfrak{B}_1}^{\;(+)}\<X\>$ up to degree $4$ follow from the commutativity.
\end{theorem}

\begin{proposition}
The polarization of $\mathcal{L}\mathcal{S}_{\mathfrak{B}_1}\<X\>$ is defined by (\ref{pol1}), (\ref{pol2}) and
\[
\{[a,b],c\}+\{[b,c],a\}+\{[c,a],b\}=0.
\]
\end{proposition}
\begin{proof}
The proof is analogous to Proposition \ref{polarization}.
\end{proof}

Let $\mathcal{L}\mathcal{S}_{\mathfrak{B}_1}^{\;\;(!)}\<X\>$ be a free algebra of the variety $\mathcal{L}\mathcal{S}_{\mathfrak{B}_1}^{\;\;(!)}$.
\begin{lemma}\label{lemmaB1}
In algebra $\mathcal{L}\mathcal{S}_{\mathfrak{B}_1}^{\;\;(!)}\<X\>$ the following identities hold:
\begin{equation}\label{id11B1}
((ab)c)d=((ac)b)d,    
\end{equation}
\begin{equation}\label{id12B1}
d(c(ab))=2d((ab)c)+((ac)d)b-2((ab)d)c,  
\end{equation}
\begin{equation}\label{id13B1}
d((ac)b)=d((ab)c)+((ac)d)b-((ab)d)c   
\end{equation}
and
\begin{equation}\label{id14B1}
c((ad)b)=d((ac)b).  
\end{equation}

\end{lemma}
\begin{proof}
It can be proved using computer algebra as software programs  Albert \cite{Albert}. 
\end{proof}

Let $\mathfrak{B}_1^{(!)}(i)$ be a set which is defined for each $i$ as follows:
\[
\mathfrak{B}_1^{(!)}(1)=\{x_{k_1}\},\;\mathfrak{B}_1^{(!)}(2)=\{x_{k_1}x_{k_2}\},
\]
and
\[
\mathfrak{B}_1^{(!)}(n)=\{((\cdots((x_{i_1}x_{i_2})x_{i_3})\cdots )x_{i_{n-1}})x_{i_n},\;x_{j_n}((\cdots((x_{j_1}x_{j_2})x_{j_3})\cdots )x_{j_{n-1}})\},
\]
where $i_1\leq i_2\leq\cdots \leq i_{n-1}$, $j_1\leq j_2\leq\cdots \leq j_{n-1}\leq j_n$ and $n\geq 3$. Also, we define 
\[
\overline{\mathfrak{B}^{(!)}}_1=\bigcup_{i\geq 1}\mathfrak{B}_1^{(!)}(i)
\]
\begin{theorem}
The set $\overline{\mathfrak{B}^{(!)}}_1$ is a basis of operad $\mathcal{L}\mathcal{S}_{\mathfrak{B}_1}^{\;\;(!)}$.
\end{theorem}
\begin{proof}
As before, we first show that the set $\overline{\mathfrak{B}^{(!)}}_1$ spans $\mathcal{L}\mathcal{S}_{\mathfrak{B}_1}^{\;\;(!)}$. For degrees up to $3$, the result is obvious. The left-symmetric and left-commutative identities give that
\begin{equation}\label{id21B1}
(ab)c=(ba)c.
\end{equation}
The left-symmetric and right-symmetric identities allow us to rewrite each monomial in $\mathcal{L}\mathcal{S}_{\mathfrak{B}_1}^{\;\;(!)}$ as a linear combination of monomials
\[
\mathcal{O}_{x_{i_1}}\cdots \mathcal{O}_{x_{i_{n-1}}}x_j,
\]
where $\mathcal{O}$ is an operator of left or right multiplication, i.e., $\mathcal{R}_xy=yx$ and $\mathcal{L}_xy=xy$. For monomial $\mathcal{O}_{x_{i_1}}\cdots \mathcal{O}_{x_{i_{n-1}}}x_j$, we consider $2$ cases:

Case 1: For $\mathcal{O}_{x_{i_1}}\ldots \mathcal{O}_{x_{i_{n-1}}}x_j$,  we consider the case $\mathcal{O}_{x_{i_1}}=\mathcal{R}_{x_{i_1}}$. By (\ref{id21B1}), we have
\[
\mathcal{R}_{x_{i_1}}\mathcal{O}_{x_{i_2}}x_j=(\mathcal{O}_{x_{i_2}}x_j)x_{i_1}=(\mathcal{L}_{x_{i_2}}x_j)x_{i_1}=(\mathcal{R}_{x_{i_2}}x_j)x_{i_1},
\]
and in general case,
\[
\mathcal{R}_{x_{i_1}}\mathcal{O}_{x_{i_2}}\cdots \mathcal{O}_{x_{i_{n-1}}}x_j=(\mathcal{R}_{x_{i_2}}\cdots \mathcal{R}_{x_{i_{n-1}}}x_j)x_{i_1}=((\cdots  (x_jx_{i_{n-1}})\cdots)x_{i_2})x_{i_1}.
\]
By (\ref{id11B1}) and (\ref{id21B1}), the generators $x_j,x_{i_{n-1}},\ldots,x_{i_2}$ can be ordered, i.e.,
\[
((\cdots((x_{i_1}x_{i_2})x_{i_3})\cdots )x_{i_{n-1}})x_{i_n}=((\cdots((x_{\sigma (i_1)}x_{\sigma (i_2)})x_{\sigma (i_3)})\cdots )x_{\sigma (i_{n-1})})x_{i_n},
\]
where $\sigma\in S_{n-1}$. The obtained monomials belong to the set $\overline{\mathfrak{B}^{(!)}}_1$. 

Case 2: For $\mathcal{O}_{x_{i_1}}\ldots \mathcal{O}_{x_{i_{n-1}}}x_j$, we consider the case $\mathcal{O}_{x_{i_1}}=\mathcal{L}_{x_{i_1}}$. By the previous case, we only need to consider monomials of the form
\begin{equation}\label{LL}
\mathcal{L}_{x_{i_1}}\cdots \mathcal{L}_{x_{i_t}} \mathcal{R}_{x_{i_{t+1}}}\cdots\mathcal{R}_{x_{i_{n-1}}}x_j=x_{i_{1}}(\cdots(x_{i_{t}}((\cdots(x_jx_{i_{n-1}})\cdots)x_{i_{t+1}}))\cdots).
\end{equation}

We denote by $\equiv$ the equality of two monomials modulo left-normed monomials, i.e., by (\ref{id12B1}), we have
\[
d(c(ab))\equiv 2d((ab)c).
\]
Now, we aim to show that for any $t$, the monomial of the form (\ref{LL}) can be written as a sum of monomials with $t=1$ and left-normed monomials.
If $t=2$, then
\begin{multline*}
\mathcal{L}_{x_{i_1}} \mathcal{L}_{x_{i_2}} \mathcal{R}_{x_{i_{t+1}}}\cdots\mathcal{R}_{x_{i_{n-1}}}x_j=x_{i_{1}}(x_{i_{2}}((\cdots(x_jx_{i_{n-1}})\cdots)x_{i_{t+1}}))\equiv^{(\ref{id12B1})}\\
2x_{i_{1}}(((\cdots(x_jx_{i_{n-1}})\cdots)x_{i_{t+1}})x_{i_{2}}).
\end{multline*}
We start the induction in $t$. For $t=k$, we have
\begin{multline*}
x_{i_{1}}(\cdots(x_{i_{t-1}}(x_{i_{t}}((\cdots(x_jx_{i_{n-1}})\cdots)x_{i_{t+1}})))\cdots)=^{(\ref{id12B1})}\\
2x_{i_{1}}(\cdots(x_{i_{t-1}}(((\cdots(x_jx_{i_{n-1}})\cdots)x_{i_{t+1}})x_{i_{t}}))\cdots)+
x_{i_{1}}(\cdots((((\cdots(x_jx_{i_{n-1}})\cdots)x_{i_{t-1}})x_{i_{t}})x_{i_{t+1}})\cdots)\\
-2x_{i_{1}}(\cdots((((\cdots(x_jx_{i_{n-1}})\cdots)x_{i_{t-1}})x_{i_{t+1}})x_{i_{t}})\cdots),
\end{multline*}
and for the obtained result, we use the inductive hypothesis.

By Case 1, we can order the generators $x_{j_1},x_{j_2},\ldots,x_{j_{n-2}}$ in
\[
x_{j_n}(((\cdots((x_{j_1}x_{j_2})x_{j_3})\cdots) x_{j_{n-2}})x_{j_{n-1}}).
\]
The identity (\ref{id13B1}) allows ordering $x_{j_{n-2}}$ and $x_{j_{n-1}}$ by module, i.e.,
\[
d((ac)b)\equiv d((ab)c),
\]
and the identity (\ref{id14B1}) orders $x_{j_{n}}$ with other generators.

To prove the linear independence of the set $\overline{\mathfrak{B}^{(!)}}_1$, we show that the compositions of rewriting rules defined the set of left-symmetric, right-symmetric, left-commutative identities and the identities (\ref{id11B1}), (\ref{id12B1}), (\ref{id13B1}), (\ref{id14B1}), (\ref{id21B1}) are trivial. The calculations are straightforward, or these long calculations can be done using the computer software \cite{DotsHij}.
\end{proof}
\begin{remark}
Indeed, the set $\overline{\mathfrak{B}^{(!)}}_1$ forms a basis of the algebra $\mathcal{L}\mathcal{S}_{\mathfrak{B}_1}^{\;\;(!)}\langle X\rangle$. Nevertheless, it would be unreasonably cumbersome to explicitly construct an algebra $B_1\langle X\rangle$ with a full multiplication table compatible with all defining identities of the variety $\mathcal{L}\mathcal{S}_{\mathfrak{B}_1}^{\;\;(!)}$.
\end{remark}

For the operad $\mathcal{L}\mathcal{S}_{\mathfrak{B}_1}^{\;\;(!)}$, we have
\[
\dim(\mathcal{L}\mathcal{S}_{\mathfrak{B}_1}^{\;\;(!)}(1))=1,\;\;\dim(\mathcal{L}\mathcal{S}_{\mathfrak{B}_1}^{\;\;(!)}(2))=2,,
\]
and
\[
\dim(\mathcal{L}\mathcal{S}_{\mathfrak{B}_1}^{\;\;(!)}(n))=n+1,
\]
where $n\geq 3$. In addition, using the computer algebra system Albert~\cite{Albert}, we obtain
\begin{center}
\begin{tabular}{c|cccccc}
 $n$ & 1 & 2 & 3 & 4 & 5 & 6 \\
 \hline $\dim(\mathcal{LS}_{\mathfrak{B}_1})(n)$ & 1 & 2 & 8 & 45 & 314 & 2533
\end{tabular}
\end{center}

In diagram~(\ref{diagram}), in the case $\mathfrak{Z}_i = \mathfrak{B}_1$, the variety
$\mathcal{A}$ may be taken to be either the variety of assosymmetric algebras or the variety
of left-commutative algebras, i.e., the variety $\mathcal{LS}_{\mathfrak{B}_1}^{\;(!)}$ is precisely the intersection of the varieties of assosymmetric and left-commutative algebras. Our computations show that, up to degree~$4$, the
classes of algebras $\mathcal{L}\mathcal{S}^{(-)}_{\mathfrak{Z}_i}$ and
$\mathcal{L}\mathcal{S}^{(+)}_{\mathfrak{Z}_i}$ coincide with the class of Lie and
commutative algebras, respectively.

\section{$\mathfrak{A}_2$-left-symmetric algebras}
In this section, we consider algebras from the variety $\mathcal{L}\mathcal{S}_{\mathfrak{A}_2}$.  Let us compute the dual operad for $\mathcal{L}\mathcal{S}_{\mathfrak{A}_2}$. The Lie-admissibility condition for $S\otimes U$ gives the defining identities of the operad $\mathcal{L}\mathcal{S}_{\mathfrak{A}_2}^{\;\;(!)}$, where $S$ is an algebra from $\mathcal{L}\mathcal{S}_{\mathfrak{A}_2}$. So, we have
\begin{multline*}
[[a\otimes u,b\otimes v],c\otimes w]=(ab)c\otimes (uv)w-(ba)c\otimes (vu)w-
c(ab)\otimes w(uv)+c(ba)\otimes w(vu)=\\
((ba)c-(ac)b+(ca)b-(bc)a+(cb)a-2b(ac)-2c(ab)-2c(ba))\otimes (uv)w\\
-(ba)c\otimes (vu)w-c(ab)\otimes w(uv)+c(ba)\otimes w(vu),
\end{multline*}
\begin{multline*}
[[b\otimes v,c\otimes w],a\otimes u]=(bc)a\otimes (vw)u-(cb)a\otimes (wv)u-
a(bc)\otimes u(vw)+a(cb)\otimes u(wv)=\\
(bc)a\otimes (vw)u-(cb)a\otimes (wv)u
-(-(ac)b+(ca)b-(bc)a+(cb)a-b(ac)-2c(ab)-2c(ba))\otimes u(vw)\\
+((ac)b-(ca)b+c(ab))\otimes u(wv)
\end{multline*}
and
\begin{multline*}
[[c\otimes w,a\otimes u],b\otimes v]=(ca)b\otimes (wu)v-(ac)b\otimes (uw)v-
b(ca)\otimes v(wu)+b(ac)\otimes v(uw)=\\
 (ca)b\otimes (wu)v-(ac)b\otimes (uw)v-
((bc)a-(cb)a+c(ba))\otimes v(wu)+b(ac)\otimes v(uw).
\end{multline*}
The sum of the same elements on the left side of the tensors, we obtain the following result: 
\begin{proposition}
The operad $\mathcal{L}\mathcal{S}_{\mathfrak{A}_2}^{\;\;(!)}$ is the alternative operad with an identity
\[
(ab)c=(ba)c
\]
\end{proposition}
\begin{proposition}
The polarization of $\mathcal{L}\mathcal{S}_{\mathfrak{A}_2}\<X\>$ is defined by (\ref{pol1}), (\ref{pol2}) and
\[
\{a,\{b,c\}\}=\{[a,b],c\}+\{[a,c],b\}+[\{b,c\},a]+2/3[[a,c],b]+1/3[a,[b,c]].
\]
\end{proposition}
\begin{proof}
The proof is analogous to Proposition \ref{polarization}.
\end{proof}

Also, from direct computations, we obtain the following results:
\begin{theorem}
All identities of the algebra $\mathcal{L}\mathcal{S}_{\mathfrak{A}_2}^{\;(-)}\<X\>$ up to degree $4$ follow from the anti-commuta\-tivity and Jacobi identities.
\end{theorem}
\begin{theorem}
All identities of the algebra $\mathcal{L}\mathcal{S}_{\mathfrak{A}_2}^{\;(+)}\<X\>$ up to degree $4$ follow from the commutative and the following identity:
\begin{multline*}
\{a, \{b, \{c, d\}\}\}+\{a,\{c,\{b, d\}\}\}+\{a, \{d,\{b, c\}\}\}+\{b,\{a, \{c, d\}\}\} \\
+\{c,\{a,\{b, d\}\}\}
+\{d,\{a, \{b, c\}\}\}+\{b,\{c,\{a, d\}\}\}+\{b,\{d,\{a, c\}\}\}\\
+\{c,\{b,\{a, d\}\}\}+\{d,\{b,\{a, c\}\}\}+\{c,\{d,\{a, b\}\}\}+\{d,\{c,\{a, b\}\}\}\\
-\{\{a, d\}, \{b, c\}\}-\{\{a, c\}, \{b, d\}\}-\{\{a, b\}, \{c, d\}\}=0.
\end{multline*}
\end{theorem}

Let $\mathcal{L}\mathcal{S}_{\mathfrak{A}_2}^{\;\;(!)}\<X\>$ be a free algebra of the variety $\mathcal{L}\mathcal{S}_{\mathfrak{A}_2}^{\;\;(!)}$.
\begin{lemma}\label{lemmaA2}
In algebra $\mathcal{L}\mathcal{S}_{\mathfrak{A}_2}^{\;\;(!)}\<X\>$ the following identities hold:
\begin{equation}\label{id11A2}
((ab)c)d=((ac)b)d,    
\end{equation}
\begin{equation}\label{id12A2}
d((ab)c)=((ab)d)c,    
\end{equation}
and
\begin{equation}\label{id13A2}
d(c(ab))=((ac)d)b,    
\end{equation}
\end{lemma}
\begin{proof}
It can be proved using computer algebra as software programs  Albert \cite{Albert}. 
\end{proof}

Let $\mathfrak{A}_2^{(!)}(i)$ be a set which is defined for each $i$ as follows:
\[
\mathfrak{A}_2^{(!)}(1)=\{x_{k_1}\},\;\mathfrak{A}_2^{(!)}(2)=\{x_{k_1}x_{k_2}\},
\]
\[
\mathfrak{A}_2^{(!)}(3)=\{(x_{l_1}x_{l_2})x_{l_3}, (x_{l_1}x_{l_3})x_{l_2}, x_{l_3}(x_{l_1}x_{l_2}), x_{l_3}(x_{l_2}x_{l_1})\}
\]
and
\[
\mathfrak{A}_2^{(!)}(n)=\{((\cdots((x_{i_1}x_{i_2})x_{i_3})\cdots )x_{i_{n-1}})x_{i_n},
\]
where $l_1\leq l_2\leq l_3$, $i_1\leq i_2\leq\cdots \leq i_{n-1}$ and $n\geq 4$. Also, we define 
\[
\overline{\mathfrak{A}^{(!)}}_2=\bigcup_{i\geq 1}\mathfrak{A}_2^{(!)}(i)
\]
\begin{theorem}
The set $\overline{\mathfrak{A}^{(!)}}_2$ is a basis of algebra $\mathcal{L}\mathcal{S}_{\mathfrak{A}_2}^{\;\;(!)}\<X\>$.
\end{theorem}
\begin{proof}
Firstly, we show that the set $\overline{\mathfrak{A}^{(!)}}_2$ spans $\mathcal{L}\mathcal{S}_{\mathfrak{A}_2}^{\;\;(!)}\<X\>$. For degrees up to $3$, the result can be verified by software programs, such as Albert \cite{Albert}. Starting from degree $4$, the inductive hypothesis and identities from Lemma \ref{lemmaA2} give the result that any monomial in $\mathcal{L}\mathcal{S}_{\mathfrak{A}_2}^{\;\;(!)}\<X\>$ can be written as a sum of left-normed monomials, i.e.,
\begin{multline*}
(((\cdots((x_{i_1}x_{i_2})x_{i_3})\cdots )x_{i_{n-1}})x_{i_n})(((\cdots((x_{j_1}x_{j_2})x_{j_3})\cdots )x_{j_{m-1}})x_{j_m})=^{(\ref{id12A2})}\\
(((\cdots((x_{j_1}x_{j_2})x_{j_3})\cdots )x_{j_{m-1}})(((\cdots((x_{i_1}x_{i_2})x_{i_3})\cdots )x_{i_{n-1}})x_{i_n}))x_{j_m},
\end{multline*}
and remain to use the inductive hypothesis.

Also, the identities from Lemma \ref{lemmaA2} and $(ab)c=(ba)c$ provide an ordering of all generators except the external one in a left-normed monomial, i.e.,
\[
((\cdots((x_{i_1}x_{i_2})x_{i_3})\cdots )x_{i_{n-1}})x_{i_n}=((\cdots((x_{\sigma (i_1)}x_{\sigma (i_2)})x_{\sigma (i_3)})\cdots )x_{\sigma (i_{n-1})})x_{i_n},
\]
where $\sigma\in S_{n-1}$.

To prove the linear independence of the set $\overline{\mathfrak{A}^{(!)}}_2$, we construct the multiplication table for the algebra $A_2\<X\>$ with the basis $\overline{\mathfrak{A}^{(!)}}_2$ and show that such an algebra $A_2\<X\>$ belongs to the variety $\mathcal{L}\mathcal{S}_{\mathfrak{A}_2}^{\;\;(!)}$. Up to degree $3$, we define multiplication in $A_2\<X\>$ that is consistent with the defining identities of $\mathcal{L}\mathcal{S}_{\mathfrak{A}_2}^{\;\;(!)}\<X\>$. Starting from degree $4$, the multiplication is defined as follows:
\begin{multline*}
(((\cdots((x_{i_1}x_{i_2})x_{i_3})\cdots )x_{i_{n-1}})x_{i_n})(((\cdots((x_{j_1}x_{j_2})x_{j_3})\cdots )x_{j_{m-1}})x_{j_m})=\\
((\cdots((x_{k_1}x_{k_2})x_{k_3})\cdots )x_{k_{n+m-1}})x_{k_{j_m}},
\end{multline*}
where $\{k_1,k_2,\ldots,k_{n+m-1}\}=\{i_1,i_2,\ldots,i_{n},j_1,j_2,\ldots,j_{m-1}\}$ and $k_1\leq k_2\leq\ldots\leq k_{n+m-1}$. It remains to check that the algebra $A_2\<X\>$ satisfies the defining identities of the variety $\mathcal{L}\mathcal{S}_{\mathfrak{A}_2}^{\;\;(!)}$.

Let
\[
a = ((\cdots((x_{i_1}x_{i_2})x_{i_3})\cdots)x_{i_{n-1}})x_{i_n},\quad
b = ((\cdots((y_{j_1}y_{j_2})y_{j_3})\cdots)y_{j_{m-1}})y_{j_m},
\]
\[
c = ((\cdots((z_{k_1}z_{k_2})z_{k_3})\cdots)z_{k_{t-1}})z_{k_t}
\]
be arbitrary left-normed monomials from the basis $\overline{\mathfrak{A}^{(!)}}_2$ of $A_2\<X\>$.  
We verify that the following identities hold in $A_2\<X\>$:
\[
(aa)b = a(ab),\qquad
(ab)b = a(bb),\qquad
(ab)c=(ba)c.
\]
By (\ref{mult1}), both sides of the identity $(aa)b = a(ab)$ can be written as a left-normed monomial whose rightmost generator is $y_{j_m}$, while all other generators are ordered. Therefore, the two expressions coincide. For the same reason, we also obtain the equalities $(ab)b = a(bb)$ and $(ab)c=(ba)c$.
\end{proof}

For the operad $\mathcal{L}\mathcal{S}_{\mathfrak{A}_2}^{\;\;(!)}$, we have
\[
\dim(\mathcal{L}\mathcal{S}_{\mathfrak{A}_2}^{\;\;(!)}(1))=1,\;\;\dim(\mathcal{L}\mathcal{S}_{\mathfrak{A}_2}^{\;\;(!)}(2))=2,\;\;\dim(\mathcal{L}\mathcal{S}_{\mathfrak{A}_2}^{\;\;(!)}(3))=4,
\]
and
\[
\dim(\mathcal{L}\mathcal{S}_{\mathfrak{A}_2}^{\;\;(!)}(n))=n,
\]
where $n\geq 4$. In addition, using the computer algebra system Albert~\cite{Albert}, we obtain
\begin{center}
\begin{tabular}{c|cccccc}
 $n$ & 1 & 2 & 3 & 4 & 5 & 6 \\
 \hline $\dim(\mathcal{LS}_{\mathfrak{A}_2})(n)$ & 1 & 2 &  8 & 44 & 285 & 1959
\end{tabular}
\end{center}

In diagram~(\ref{diagram}), in the case $\mathfrak{Z}_i = \mathfrak{A}_2$, the variety
$\mathcal{A}$ may be taken to be either the variety of alternative algebras or the variety
of algebras whose left-center is the commutator, i.e., the variety $\mathcal{LS}_{\mathfrak{A}_2}^{\;(!)}$ is precisely the intersection of the varieties of alternative algebras and the variety
of algebras whose left-center is the commutator. Analogical algebras with a left-center were considered in \cite{DzhumaLeib+, SarAbd}. Our computations show that, up to degree~$4$, the
classes of algebras $\mathcal{L}\mathcal{S}^{(-)}_{\mathfrak{Z}_i}$ and
$\mathcal{L}\mathcal{S}^{(+)}_{\mathfrak{Z}_i}$ coincide with the class of Lie and
commutative algebras with the additional identity of the form
\begin{equation}\label{Jord}
\sum_{i_3<i_4} \{x_{i_1},\{x_{i_2},\{x_{i_3},x_{i_4}\}\}\}-\sum_{i_1<i_2,i_3<i_4,i_2<i_4}\{\{x_{i_1},x_{i_2}\},\{x_{i_3},x_{i_4}\}\}=0,
\end{equation}
respectively.

From the table given below, we see that the algebra $\mathcal{L}\mathcal{S}_{\mathfrak{A}_2}\langle X\rangle$ is not isomorphic to either of the other two algebras. From the dimensions of $\mathcal{L}\mathcal{S}_{\mathfrak{A}_1}\langle X\rangle$ and $\mathcal{L}\mathcal{S}_{\mathfrak{B}_1}\langle X\rangle$ of degree $6$, it follows that these algebras are not isomorphic.

Let us summarise some of the obtained results in the following table:

\begin{table}[h!]
\centering
\begin{tabular}{|c|p{3cm}|p{3cm}|p{3cm}|p{4cm}|}
\hline
\textbf{Type} 
& \textbf{Dual operad $\mathcal{L}\mathcal{S}_{\mathfrak{Z}_i}^{\;(!)}$} 
& \textbf{Dimensions of operad up to degree $5$} 
& \textbf{Commutator (up to degree $4$)} 
& \textbf{Anti-commutator (up to degree $4$)} \\
\hline
$\mathcal{L}\mathcal{S}_{\mathfrak{A}_1}$ 
& Alternative and $\frac{1}{2}$-Zinbiel 
& dim($\mathcal{L}\mathcal{S}_{\mathfrak{A}_1}$): $\;\;\;\;\;\;\;$ 1, 2, 8, 45, 314
dim($\mathcal{L}\mathcal{S}_{\mathfrak{A}_1}^{\;(!)}$): $\;\;\;\;\;\;\;$ 1, 2, 4, 5, 5
& Lie algebra 
& Commutative algebra \\
\hline
$\mathcal{L}\mathcal{S}_{\mathfrak{B}_1}$ 
& Assosymmetric and left-commutative 
& dim($\mathcal{L}\mathcal{S}_{\mathfrak{B}_1}$): $\;\;\;\;\;\;\;$ 1, 2, 8, 45, 314 
dim($\mathcal{L}\mathcal{S}_{\mathfrak{B}_1}^{\;(!)}$): $\;\;\;\;\;\;\;$ 1, 2, 4, 5, 6
& Lie algebra 
& Commutative algebra \\
\hline
$\mathcal{L}\mathcal{S}_{\mathfrak{A}_2}$ 
& Alternative and left-center is commutator 
& dim($\mathcal{L}\mathcal{S}_{\mathfrak{A}_2}$): $\;\;\;\;\;\;\;$ 1, 2, 8, 44, 285 
dim($\mathcal{L}\mathcal{S}_{\mathfrak{A}_2}^{\;(!)}$): $\;\;\;\;\;\;\;$ 1, 2, 4, 4, 5
& Lie algebra 
& Commutative algebras with an additional identity of degree $4$ which is (\ref{Jord}) \\
\hline
\end{tabular}
\end{table}


\begin{thebibliography}{99}

\bibitem{Albert} Albert version 4.0M6; https://web.osu.cz/$\sim$Zusmanovich/soft/albert/

\bibitem{Mash1}
S. Bouarroudj, F. Mashurov, On Zinbiel and Tortkara Superalgebras, Bulletin of the Brazilian Mathematical Society, 2025, 56(3), 49.

\bibitem{DotsHij}
V. Dotsenko, W. Heijltjes. Gröbner bases for operads, http://irma.math.unistra.fr/dotsenko/operads.html, 2019.

\bibitem{KazMat}
Y. Duisenbay, B. Sartayev, A. Tekebay, 3-nil alternative, pre-Lie, and assosymmetric operads, 2024, Kazakh Mathematical Journal, 24(4), 3751.
https://doi.org/10.70474/41eaqa09


\bibitem{DzhumaLeib+}
A. S. Dzhumadildaev,
Jordan elements and left-center of a free Leibniz algebra, Electr. Res. Announcements in Mathem. Sci., 18, 2011, 31-49.

\bibitem{assos1}
A. S. Dzhumadil'daev, B. K. Zhakhayev, S. A. Abdykassymova, Assosymmetric operad, Communications in Mathematics, 2022, 30(1), 175-190.

\bibitem{Dzhuma}
A. Dzhumadiĺdaev, P. Zusmanovich, The alternative operad is not Koszul, Experimental Mathematics, 2011, 20(2), 138-144.


\bibitem{Gerst}
M. Gerstenhaber, 
The cohomology structure of an associative ring, 
Ann. Math., 1963, 78, 267-288.

\bibitem{GK94}
V. Ginzburg, M. Kapranov,
Koszul duality for operads,
Duke Mathematical Journal, 76, 1994, 1, 203--272.

\bibitem{Mash2}
N. Ismailov, F. Mashurov, N. Smadyarov, Defining identities for mono and binary Zinbiel algebras, Journal of Algebra and Its Applications, 2023, 22(8), 2350165.

\bibitem{KolMashSar}
P. Kolesnikov, F. Mashurov, B. Sartayev, On Pre-Novikov Algebras and Derived Zinbiel Variety, Symmetry, Integrability and Geometry: Methods and Applications (SIGMA), 2024, 20, 17.


\bibitem{As1}
I. Kaygorodov, Non-associative algebraic structures: classification and structure, Communications in Mathematics, 2024, 32(3), 1-62.


\bibitem{As2}
I. Kaygorodov, M. Khrypchenko, P. Páez-Guillán, The geometric classification of non-associative algebras: a survey, Communications in Mathematics, 2024, 32(2 Special issue), 185-284.

\bibitem{binaryperm}
A. Kunanbayev, B. Sartayev, Binary perm algebras and alternative algebras, Communications in Algebra, 54, 2026, 299-307.

\bibitem{Kosz}
J.-L. Koszul, 
Domaines born\'es homog\`enes et orbites de groupes 
de transformations affines,
Bull. Soc. Math. Fr., 1961, 89, 515-533.


\bibitem{MalcAs1}
A. Kunanbayev, B. Sartayev, 4-type subvarieties of the variety of associative algebras, 2025, https://arxiv.org/abs/2506.02471.

\bibitem{Mal'cev}
A. I. Mal’cev, On algebras defined by identities (Russian), Mat. Sbornik N. S., 26(68), 1950, 19–33.

\bibitem{MarklRemm}
M. Markl, E. Remm, Algebras with one operation including Poisson and other Lie-admissible algebras, Journal of Algebra, 299(1), 2006, 171–189.


\bibitem{-1-1}
S. Pchelintsev, Identities defining a certain variety of right-alternative algebras, Mathematical Notes, 20(2), 1976, 651–659.


\bibitem{MalcAs2}
A. Ydyrys, B. Sartayev, Free Algebras in Mal'cev-Type Subvarieties of Associative Algebras, 2025, https://arxiv.org/abs/2510.04181.

\bibitem{SarAbd}
B. Sartayev, A. Ydyrys, Free products of operads and Gröbner base of some operads, Proceedings - 2023 17th International Conference on Electronics Computer and Computation, ICECCO 2023, 2023.

\bibitem{Vin}
E. B. Vinberg, 
Homogeneous cones, 
Sov. Math. Dokl., 1960, 1, 787-790.

\bibitem{Artin}
K. A. Zhevlakov, A. M. Slin'ko, I. P. Shestakov, A. I. Shirshov, Rings that are nearly associative,
Translated from the Russian by Harry F. Smith. Pure and Applied Mathematics, 104. Academic Press, Inc. [Harcourt Brace Jovanovich, Publishers], New York-London, 1982. xi+371 pp.



\end{thebibliography}
\end{document}